\documentclass{amsart}

\usepackage{caption}
\usepackage{tikz-cd}
\usepackage{subcaption}
\usepackage[percent]{overpic}

\usepackage{fullpage}
\usepackage{amsmath}
\usepackage{amsfonts}
\usepackage{amssymb}
\usepackage{mathtools}
\usepackage{stmaryrd}
\usepackage{bm}
\usepackage{amsthm}
\usepackage{xcolor} 
\usepackage{lipsum}

\newtheorem{innercustomgeneric}{\customgenericname}
\providecommand{\customgenericname}{}
\newcommand{\newcustomtheorem}[2]{%
  \newenvironment{#1}[1]
  {%
   \renewcommand\customgenericname{#2}%
   \renewcommand\theinnercustomgeneric{##1}%
   \innercustomgeneric
  }
  {\endinnercustomgeneric}
}

\newcustomtheorem{reference}{Reference}
\newcustomtheorem{customcor}{Corollary}
\newcustomtheorem{open-problem}{Open Problem}

\newtheorem{theorem}{Theorem}[section]
\newtheorem{lemma}[theorem]{Lemma}

\newtheorem{proposition}[theorem]{Proposition}
\usepackage{lineno}
\usepackage{xifthen}

\providecommand{\customgenericname}{}

\newcustomtheorem{customthm}{Theorem}

\usepackage[capitalize]{cleveref}
\crefname{equation}{Eq.}{Eqs.}
\Crefname{equation}{Equation}{Equations}

\newcommand\sci[2][1]{\ensuremath{\ifthenelse{\equal{#1}{1}}{}{#1\times}10^{#2}}}

\usepackage[obeyFinal]{todonotes}
\usepackage[color]{changebar}
\usepackage{soul}                     
\usepackage{ifdraft}

\title{Gassner and Burau representations over $\mathbb{Z}_p$-modules}
\author{Vasudha Bharathram}
\date{\today}

\begin{document}

\begin{abstract}
We study two classical representations of Artin's braid group and their modulo $p$ reductions. We use topological methods to show that the Gassner representation $\tau_n: B_n\to\text{GL}_n(\mathbb{Z}[t_1^{\pm 1}, \ldots, t_n^{\pm 1}])$ is faithful for all $n$, and furthermore that it is faithful modulo $p$ for all integers $p>1$. We then give a novel proof that the Burau representation of $B_3$ is faithful modulo $p$ for all $p>1$, and suggest applications to the modulo $p$ Burau representation for higher braid groups. 
\end{abstract}

\subjclass[2020]{20F36, 20F65, 20C99}
\keywords{Braid group, Gassner representation, Burau representation}

\maketitle

\section{Introduction}

The classical braid groups $B_n$ were introduced by Artin in \cite{artin1925} and are now ubiquitous in mathematics. They have been found to have deep connections with a vast range of fields, from knot theory to statistical mechanics, and in addition admit a fascinating theory that is of great intrinsic interest. Over the past century, one of the problems in the area that has received particular attention is that of whether the braid groups are linear, a question that was ultimately settled in the positive by seminal work of Bigelow \cite{Big00} and Krammer \cite{Kra}. The Lawrence-Krammer representation of $B_n$ in $\text{GL}_{n(n-1)/2}(\mathbb{Z}[t^{\pm 1},q^{\pm 1}])$ that they both proved to be faithful is, at the time of this writing, the only known faithful linear representation of the braid groups. 

Before the Lawrence-Krammer representation was studied, however, the main contender for a faithful representation was the Burau representation $\rho_n: B_n\to\text{GL}_{n-1}(\mathbb{Z}[t, t^{-1}])$, introduced by Werner Burau in \cite{burau}. This representation has long been known to be faithful for $n\leq 3$ \cite{magnuspeluso} and was shown in the 1990s to be unfaithful for $n\geq 5$ \cite{moody, longpatton, bigelow}. In \cite{gassner} a natural generalisation of the Burau representation to the pure braid group $PB_n$, the {\it Gassner representation} $\tau_n: PB_n\to\text{GL}_{n-1}(\mathbb{Z}[t_1^{\pm 1}, \ldots, t_n^{\pm 1}])$, was introduced, and its faithfulness has remained unknown for $n\geq 4$. Our main result is an answer to this question, and an extension to an infinite family of faithful representations of each pure braid group $PB_n$ over $\mathbb{Z}_p[t_1^{\pm 1},\ldots,t_n^{\pm 1}]$:

\begin{customthm}{A}
    The Gassner representation $\tau_n$ is faithful for all $n$.
\end{customthm}

\begin{customthm}{B}
    The Gassner representation $\tau_n$ is faithful modulo $p$ for all $n$ and all integers $p>1$.
\end{customthm}

We will also apply the topological methods we used to study the Gassner representation to prove corresponding results regarding the Burau representation of $B_3$, giving a new proof that $\rho_3\otimes \mathbb{Z}_p$ is faithful for all $p>1$. While this result has been established, i.e. via \cite{Lee-Song}, we hope that the methods we use here will be expanded upon to investigate the representations $\rho_n\otimes \mathbb{Z}_p$ for higher $n$, which seems to be a very interesting and open area for future study.

{\bf Acknowledgments.} The author is deeply indebted to Joan Birman for guiding her since high school, for introducing her to this subject and to the world of mathematics, for many hours of helpful conversations without which the results in this paper would not have been proved, and for too many other things to list. The author is also very thankful to Tara Brendle for helpful conversations and suggestions, and for pointing out an error in an earlier version of this paper. 

\section{The Burau and Gassner representations}

\subsection{Topological definitions} 

The (reduced) Burau representation of the braid group $B_n$, given by $\rho_n: B_n \to \text{GL}_{n-1}(\mathbb{Z}[t, t^{-1}])$, was first introduced in 1936 by Werner Burau in \cite{burau}. The representation is now ubiquitous in the field, and the Gassner representation is a natural generalisation to the pure braid group $PB_n$. While both representations are generally defined in terms of their action on standard generating sets of $B_n$ and $PB_n$, we prefer here to consider their action on the mapping class group and pure mapping class group of the $n$-punctured disc. 

With that in mind, let $D$ be a disc and let $p_1, \ldots, PB_n$ be distinct points on that disc. Let $D_n$ denote the $n$-punctured disc $D\backslash \{p_1, \ldots, PB_n\}$ with basepoint $p_*$ on $\partial D$.  The braid group $B_n$ is the mapping class group of $D_n$. 

Now, let $x_1, \ldots, x_n$ denote the free generators of $\pi_1(D_n, p_*)$, and consider the mapping $\epsilon: \pi_1(D_n,p_*)\to \mathbb{Z}$ that takes a word in $x_1, \ldots x_n$ to the sum of its exponents. Denote by $\tilde{B_n}$ the covering space associated to the kernel of $\epsilon$. The group of covering transformations of $\tilde{B_n}$ is isomorphic to $\mathbb{Z}$, which we denote as a multiplicative group generated by $t$. The homology group $H_1(\tilde{B_n})$ is a free $\mathbb{Z}[t, t^{-1}]$-module of rank $n-1$.

Then, let $\varphi$ be an autohomeomorphism of $D_n$, i.e. an element of $B_n$. We can lift $\varphi$ to a map $\tilde{\varphi}: \tilde{B_n}\to\tilde{B_n}$ that fixes the fibre over $p_*$ pointwise, thus inducing a $\mathbb{Z}[t, t^{-1}]$-module automorphism $\tilde{\varphi}_*$ of $H_1(\tilde{B_n})$; the reduced Burau representation $\rho_n$ takes $\varphi$ to $\tilde{\varphi}_*$. 

\begin{figure}[htpb!]\centering
\begin{subfigure}[c]{.43\textwidth}\centering
    \begin{overpic}[scale=.4]{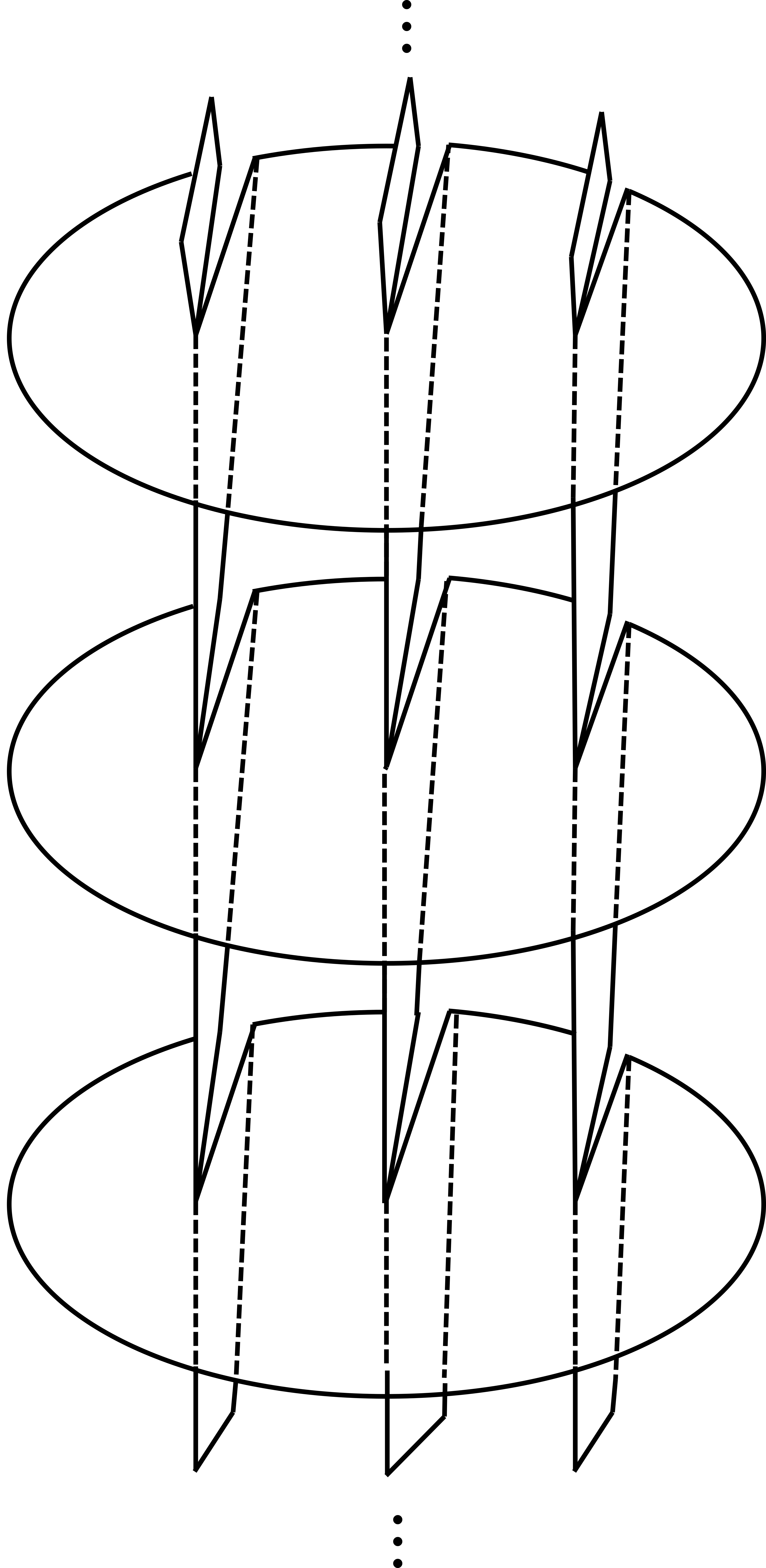}
    \end{overpic}
    \caption{The covering space $\tilde{B}_3$; the covering transformation $t$ shifts the entire space up one level.}
\end{subfigure}\hspace{5mm}
\begin{subfigure}[c]{.43\textwidth}\centering
    \begin{overpic}[scale=1]{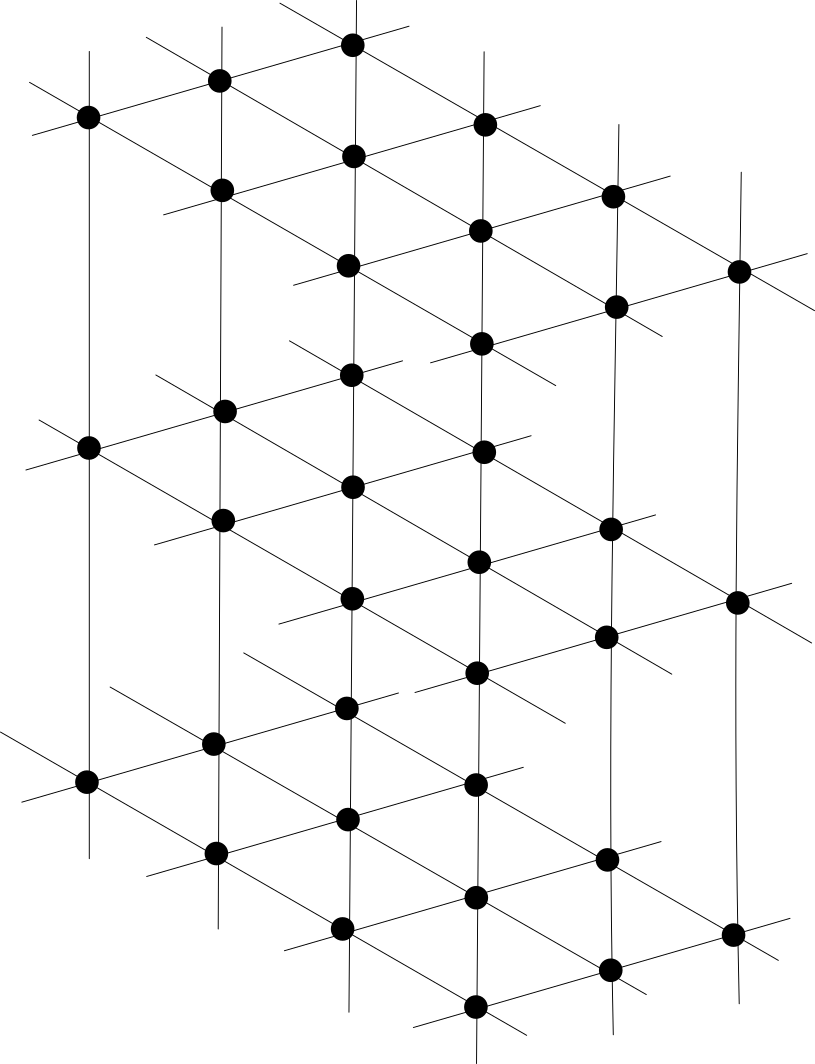}
    \end{overpic}
    \caption{A deformation retract of the covering space $\tilde{G}_3$; the covering transformation $t_i$ shifts the entire space one level in the $i$-th dimension.}
\end{subfigure}
\caption{The Burau and Gassner covering spaces}
\label{fig:covering-space}
\end{figure}

We may define the Gassner representation $\tau_n$ of the braid group analogously; we simply modify the definition of the Burau representation to take a word in $x_1, \ldots, x_n$ to an $n$-tuple containing the sum of the exponents of each $x_i$, so that $\epsilon$ takes $\pi_1(D_n, p_*)$ to $\mathbb{Z}^n$ rather than $\mathbb{Z}$. Then, letting $\tilde{G}_n$ denote the corresponding covering space, any braid $\phi\in PB_n$ lifts to a map $\tilde{\phi}: \tilde{G}_n\to \tilde{G}_n$, which induces a $\mathbb{Z}[t_1, t_1^{-1}, \ldots, t_n, t_n^{-1}]$ automorphism $\tilde{\phi}_*$ of $H_1(\tilde{G}_n)$. The Gassner representation then takes $\phi$ to $\tilde{\phi}_*$

\subsection{Burau and Gassner polynomials} Let $\alpha_0$ and $\beta_0$ be an arbitrary but fixed choice of disjoint simple arcs on $D_n$, where $\alpha_0$ runs from $p_1$ to $p_2$ and $\beta_0$ runs from $p_*$ to $p_3$.  Choose any element $\phi\in B_n$ that fixes the punctures and the base point, and set $\alpha = \alpha_0$ and $\beta = \phi(\beta_0)$.  Assume that $\beta$ has been modified by isotopy so that the arcs intersect minimally and transversely. 

Then, the {\it Burau polynomial}, well-defined up to multiplication by an arbitrary power of $t$, is given by: 
\begin{align*} 
\mathcal B_{\phi}(t)=\sum_{k\in\mathbb{Z}} (t^k\tilde{\alpha}, \tilde{\beta})t^k
\end{align*}
where $(t^k\tilde{\alpha}, \tilde{\beta})$ denotes the algebraic intersection number of $t^k\tilde{\alpha}$ and $\tilde{\beta}$ in $\tilde{B_n}$. This polynomial was first introduced by Moody in \cite{moody}. Similarly, the {\it Gassner polynomial} is given by 
\begin{align*} 
\mathcal G_{\phi}(t_1, \ldots, t_n)=\sum_{k_1, \ldots, k_n\in\mathbb{Z}} (t_1^{k_1}\ldots t_n^{k_n}\tilde{\alpha}, \tilde{\beta})t_1^{k_1}\ldots t_n^{k_n}
\end{align*}

The polynomials $\mathcal B_{\phi}(t)$ and $\mathcal G_{\phi}(t_1, \ldots, t_n)$ give necessary and sufficient conditions for the faithfulness of $\rho_n$ and $\tau_n$ for any value of $n$, and this result has been the main tool used to study the faithfulness of the Burau representation to date. The fact that the Gassner polynomial gives a necessary and sufficient condition for determining $\tau_n$'s faithfulness does not seem to have been used in the literature, although it was originally proved alongside the corresponding result regarding the Burau representation in \cite{moody}. Since we will be particularly interested in representations that factor through the Gassner representation, i.e. modulo $p$ reductions, we will state a slight generalisation of Moody's criteria regarding representations that factor through $\tau_n$ via ring homomorphisms from $\mathbb{Z}[t_1^{\pm1},\ldots,t_n^{\pm 1}]$. 

\begin{theorem}\label{modp-thm}
    If $\pi: \mathbb{Z}[t_1^{\pm 1}, \ldots, t_n^{\pm 1}]\to R$ is a ring homomorphism, $(t_1\ldots t_n)^k \not \in \ker(\pi)$ for all $k\neq 0$, and $\pi': \text{GL}_n(\mathbb{Z}[t_1^{\pm 1}, \ldots, t_n^{\pm 1}])\to \text{GL}_n(R)$ is the induced homomorphism of general linear groups, then the following are equivalent:
    \begin{enumerate}
        \item The composition $\tau_n\circ \pi': PB_n\to \text{GL}_{n-1}(R)$ is faithful
        \item The image of the Gassner polynomial $\pi(\mathcal{G}_{\phi}(t_1, \ldots, t_n))$ is non-zero for all $\phi\in B_n$ such that the geometric intersection number $\hat{\iota}(\alpha_0, \phi(\beta_0))>0$. 
    \end{enumerate}
\end{theorem}
\begin{proof}
    This follows almost exactly as in \cite{bigelow} and \cite{moody}; the only distinction comes from the fact that since we are working with $\tau_n\circ \pi'$ rather than $\tau_n$ itself, we need an extra criteria to rule out the possibility that powers of the Garside braid $\Delta^2$ (the only non-trivial pure braids that give non-intersecting $\alpha$ and $\beta$) are not in $\ker(\tau_n\circ \pi')$. However, noting that $\Delta$ acts on $H_1(\tilde{G}_n)$ as multiplication by $t_1\ldots t_n$, our hypothesis that $(t_1\ldots t_n)^k \not \in \ker(\pi)$ gives us that no powers $\Delta^{2k}$ are in the kernel of $\tau_n\circ \pi'$, and we can then follow Bigelow and Moody exactly. 
\end{proof}

We point out that by setting $R=\mathbb{Z}[t,t^{-1}]$ and $\pi: \mathbb{Z}[t_1^{\pm 1}, \ldots, t_n^{\pm 1}]\to R$ to be the map sending each $t_i\to t$, we recover Moody, Long-Paton, and Bigelow's faithfulness condition for the Burau representation.

Our strategy, now, will be to impose restrictions on $\mathcal{G}_{\phi}(t_1, \ldots, t_n)$, so that we can obtain faithful representations of the braid groups over $\mathbb{Z}_p[t_1^{\pm 1},\ldots t_n^{\pm 1}]$ via applications of Theorem \ref{modp-thm}. Before we do that, however, we will give a planar interpretation of the polynomials that will be fundamental to our arguments; this will allow us to work on the disc $D_n$ itself rather than the covering space $\tilde{G}_n$, making for easier topological arguments. 



\subsection{Planar interpretation of polynomials} We will begin by giving a planar interpretation of the Burau polynomial $\mathcal{B}_{\phi}$, which will then easily generalise to the Gassner polynomial $\mathcal{G}_{\phi}(t_1, \ldots, t_n))$. These observations will then allow us to study $\rho_n\otimes \mathbb{Z}_p$ and $\tau_n\otimes \mathbb{Z}_p$ on the disc itself. 

Notice, then, that each point $q_i$ at which $\beta$ crosses $\alpha$ contributes a monomial $\pm t^{k_i}$ to the Burau polynomial $\mathcal B_{\phi}(t)$. The exponent $k_i$ is such that $\tilde{\beta}$ and $t^{k_i} \tilde{\alpha}$ cross at a lift of $q_i$, and the sign of the monomial is the sign of that crossing. We choose our lifts and sign conventions such that the first point at which $\beta$ crosses $\alpha$ is assigned the monomial $+t^0$.

In Figure \ref{poly-calc}, the sign of the monomial at a crossing $p_i$ will be positive if $\beta$ is directed downwards at $p_i$ and negative if $\beta$ is directed upwards at $p_i$. The exponents of the monomials can be computed as follows:

Then, let $q_i, q_{i+1}\in \alpha\cap\beta$ and let $k_i$ and $k_{i+1}$ be the exponents of the monomials at $q_i$ and $q_{i+1}$. Let $\alpha_{\mu_i}$ and $\beta_i$ be the arcs from $q_i$ to $q_{i+1}$ along $\alpha$ and $\beta$ respectively and suppose that $\alpha_{\mu_i}\cap\beta_i=\{q_i, q_{i+1}\}$. Let $k$ be such that $\alpha_{\mu_i}\cup\beta_i$ bounds a $k$-punctured disc. Then $|k_{i+1}-k_i|=k$. If $\beta_i$ is directed counterclockwise around the $k$-punctured disc, then $k_{i+1}\geq k_i$, otherwise $k_i\geq k_{i+1}$. 

We can thus progress along $\beta$, calculating the exponent at each crossing based on the exponent at the previous crossing, as is shown below. We will label the intersections $q_1, q_2, \ldots, q_m$ of $\alpha$ and $\beta$ according to the order in which they appear on $\beta$.

\begin{figure}[htpb!]\centering
\begin{overpic}[scale=.9]{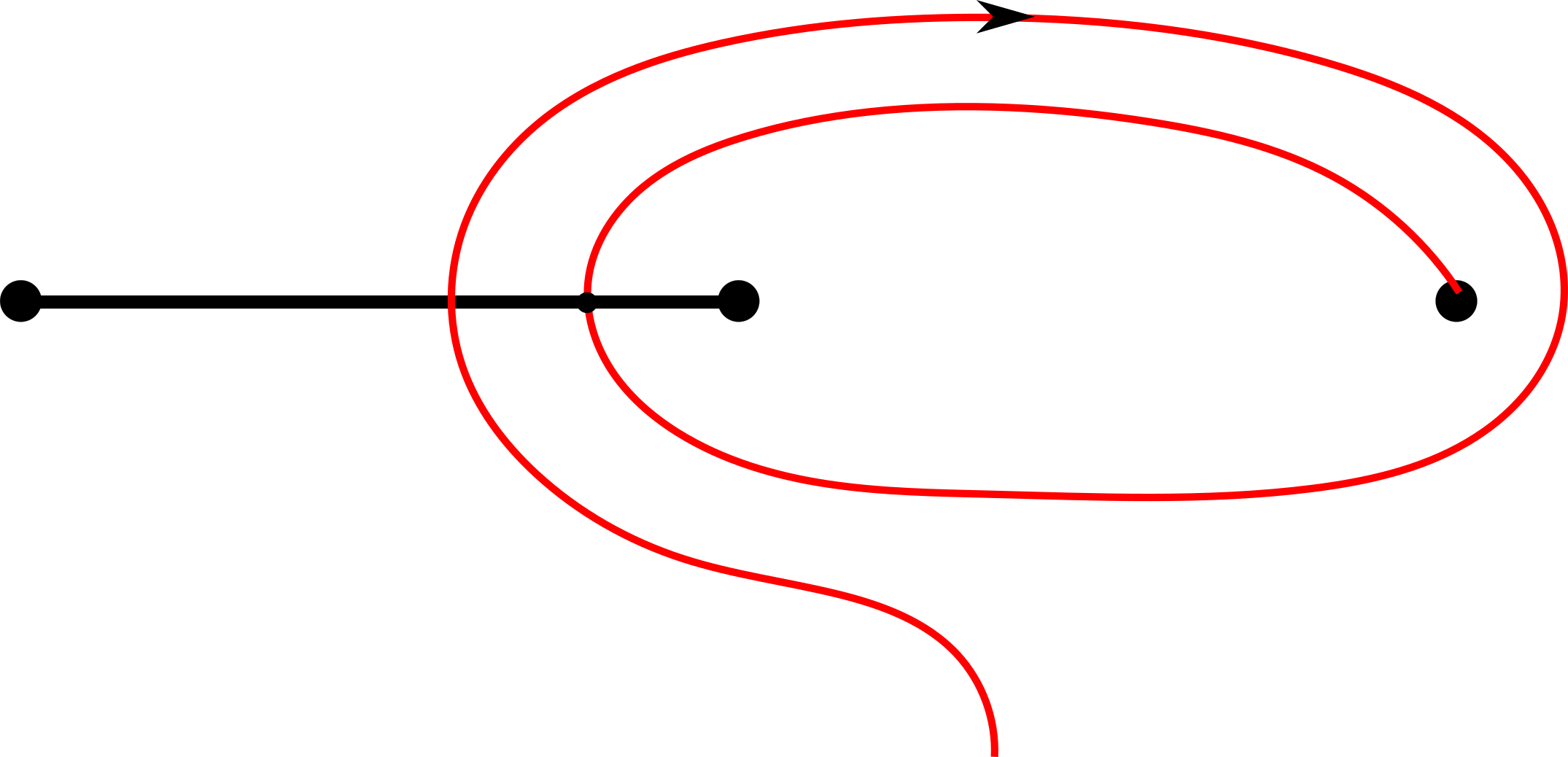}
    \put (15,24) {$\alpha$}
    \put (68, 9) {$\beta$}
   \put (23,31) {\tiny $t^0$}
   \put (32,31) {\tiny $t^2$}
   \put (0,25) {\tiny $p_1$}
  \put (45,25) {\tiny $p_2$}
   \put (89,25) {\tiny $p_3$}
\end{overpic}
\caption{Arcs $\alpha$ and $\beta$ corresponding to a Burau polynomial of $1+t^2$}
\label{poly-calc}
\end{figure}

The Gassner polynomial can be similarly calculated, where the difference in the exponents of a term $t_j$ in the monomials at $q_i$ and $q_{i+1}$ is either 0 or $\pm 1$ depending on whether or not $p_j\in \alpha_{\mu_i}\cup\beta_i$. In the example in Figure \ref{poly-calc}, the Gassner polynomial is thus given by $1+t_2 t_3$. 

\section{Faithful representations of $B_n$}

We are now ready to investigate the form of the Gassner polynomial. Our general idea will be to use the planar interpretation of the polynomial we gave in the previous section to express the Gassner polynomial as a sum of monomials at $\alpha$ and $\beta$'s crossings, i.e.
\begin{align*}
    \mathcal{G}_\phi(t_1, \ldots, t_n)=\sum_{i=0}^m a_i t_1^{k_{1, i}}\ldots t_n^{k_{n, i}}, \text{ where each }a_i\in\{-1, 1\}.
\end{align*}

According to Theorem \ref{modp-thm}, to study the faithfulness of $\tau_n$ it suffices to understand the forms of the distinct values $i, j$ at which each $k_{x, i}=k_{x, j}$, in which case we will say that $(i, j)$ is a {\it matching pair}. 

\begin{figure}[htpb!]
\begin{subfigure}[c]{.43\textwidth}
\begin{overpic}[scale=.8]{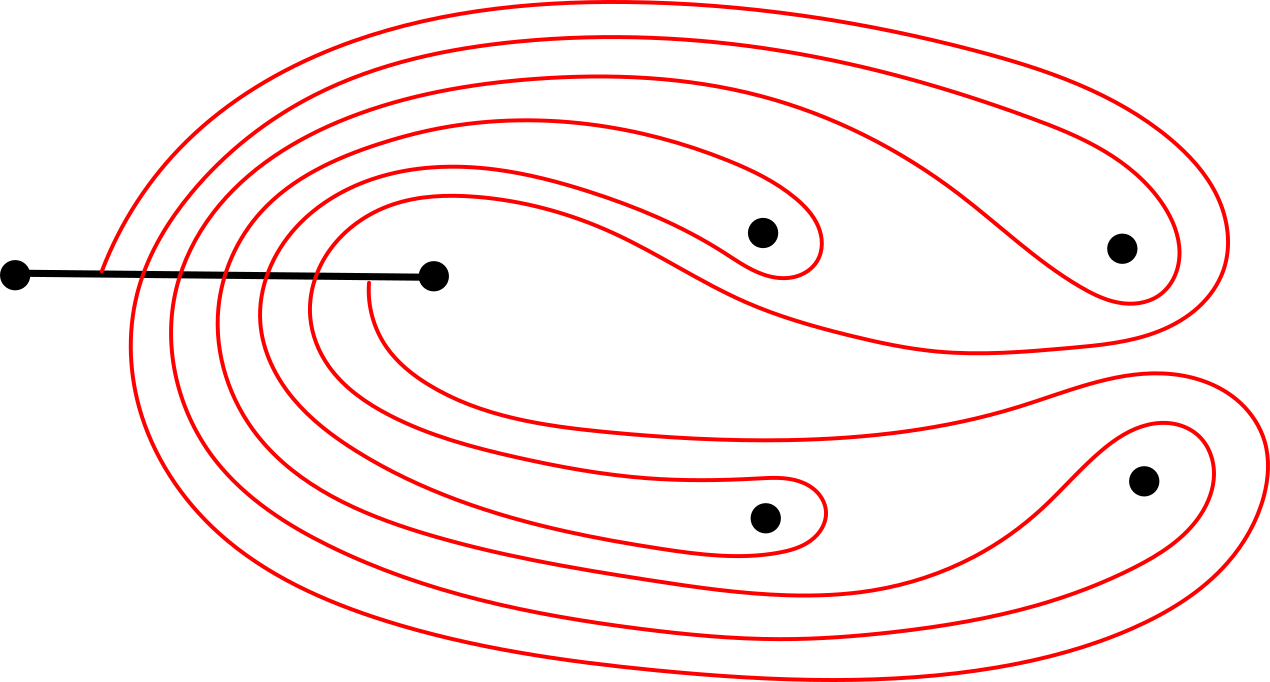}
\end{overpic}
\caption{A matching pair $(1, 7)$ where $\beta$ has the same sign at $q_1$ and $q_7$.}
\end{subfigure}
\begin{subfigure}[c]{.43\textwidth}
\begin{overpic}[scale=.8]{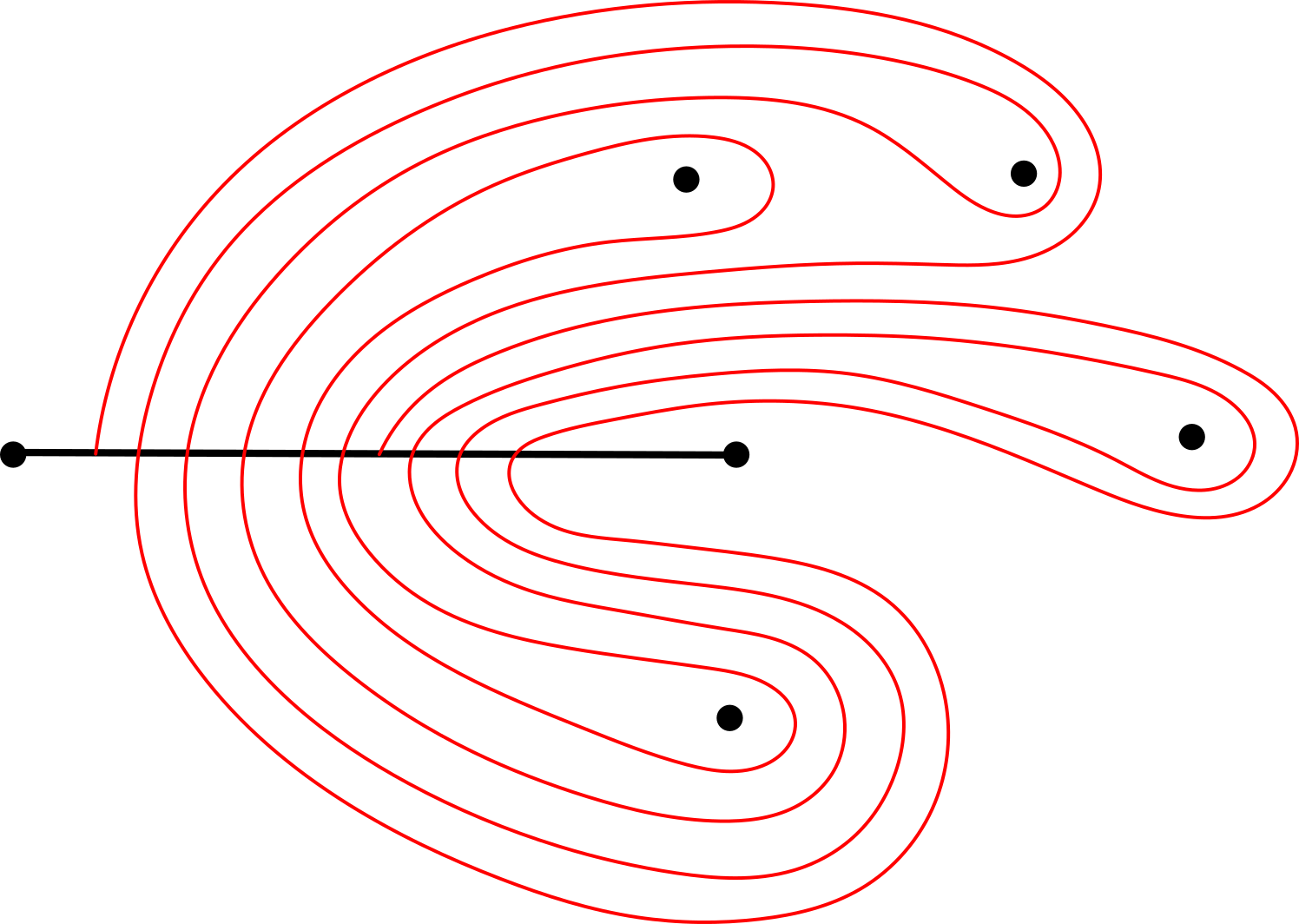}
\end{overpic}
\caption{A matching pair $(1, 10)$ where $\beta$ has different signs at $q_1$ and $q_{10}$.}
\end{subfigure}
\caption{Examples of matching pairs.}
\label{fig:matching-pairs}
\end{figure}

The following Lemma is then crucial.

\begin{lemma}[Key Lemma]
If $\alpha$ and $\beta$ have $m>0$ crossings, then, there is some $1\leq t\leq m$ such that there is no $1<l\leq m$ so that $(t, l)$ is also a matching pair. 
\end{lemma}
\begin{proof}
    Define an ordering on $\alpha$ and $\beta$'s crossings $q_t$ wherein $q_v<q_{v'}$ if $q_v$ occurs before $q_{v'}$ along $\alpha$, and assume $\alpha$ is oriented so that $q_1<q_m$. Then, let $t_0=1, t_{r}=m$, and let $t_1, \ldots, t_r$ be such that $t_0 < t_1 < \ldots < t_r$ and $q_{t_0} < q_{t_1} <\ldots < q_{t_r}$. Assume we have chosen the $t_v$'s so that for each $v$ there is no $w>t_v$ so that $q_{t_v}<q_w<q_{t_{v+1}}$. Then, let $t_0^*=1, t_{r'}^*=m$, and let $t_1^*,\dots,t_{r'}^*$ be defined similarly, so that $t_0^*<\ldots<t_{r'}^*$ and $q_{t_0^*} <\ldots < q_{t_{r'}^*}$, except that we instead choose the $t_v^*$'s so that for each $v$ there is no $w<t_{v+1}$ so that $q_{t_{v-1}^*}<q_w<q_{t_{v}^*}$. 
    
    Then, taking $\alpha(a,b)$ to denote the portion of $\alpha$ that runs from $q_a$ to $q_b$, we can for each $v$ define a {\it simple} closed curve 
    \begin{align*}
        \gamma(t_v, t_{v+1}) := (\beta_{t_v}\cup\ldots\cup\beta_{t_{v+1}-1})\cup\alpha(t_v,t_{v+1})
    \end{align*}
    Since $\alpha$ and $\beta$ are in minimal position, each $\gamma(t_v, t_{v+1})$ must contain a non-empty set of punctures $\mathcal{P}(\gamma(t_v,t_{v+1}))$. Each $\gamma(t_v,t_{v+1})$ also may contain discs $\alpha_w\cup\beta_w, t_v\leq w\leq t_{v+1}$ that bound certain punctures that are not contained in the interior of $\gamma(t_v,t_{v+1})$. These punctures will be of particular interest to us. With that in mind, for any $\gamma(t_v,t_{v+1})$, let 
    \begin{align*}
        \mathcal{C}(\gamma(t_v,t_{v+1})):=\bigcup_{t_v\leq w<t_{v+1}}\mathcal{P}(\alpha_w\cup\beta_w)
    \end{align*}

    We may repeat this construction to find curves $\gamma(t_v^*,t_{v+1}^*)$ and sets $\mathcal{C}(\gamma(t_v^*,t_{v+1}^*))$. We will show that we can reduce to the case where there is some puncture $p_x$ that is in exactly one $\mathcal{C}(\gamma(t_v,t_{v+1}))$ or exactly one $\mathcal{C}(\gamma(t_v^*,t_{v+1}^*))$. 

    To begin to see this, consider any puncture $p_y \in\mathcal{P}(\gamma(t_{v'-1},t_{v'}))$ that lies in the interior of the simple closed curve $\gamma(t_{v'-1},t_{v'})$. Then, there is some $\alpha_z\cup\beta_z$, $t_{v'-1}\leq z<t_{v'}$ that bounds $p_y$. If there is some $v \neq v'$ so that $p_y\in\mathcal{C}(\gamma(t_v,t_{v+1}))$, then we have some $t_v\leq w<t_{v+1}$ so that $p_y\in\mathcal{P}(\alpha_w\cup\beta_w)$. It then follows that the interior of one of $\alpha_{z-1}\cup\beta_{z-1}$ and $\alpha_w\cup\beta_w$ is contained in the other. 

    Then, if we assume that we have chosen $w$ so that any other $t_v\leq w'<t_{v+1}$ with $p_y\in\mathcal{P}(\alpha_{w'}\cup\beta_{w'})$ has $\alpha_{w'}\cup\beta_{w'}$ contained in $\alpha_{w}\cup\beta_{w}$, and similarly for the choice of $z$, we may use the nesting order to define a partial order on the indices $1,\dots,r$ wherein we say that $v' >_N v$ if $\alpha_w\cup\beta_w$ is contained in $\alpha_z\cup\beta_z$ and $v' <_N v$ if $\alpha_z\cup\beta_z$ is contained in $\alpha_w\cup\beta_w$. 

    Since $r$ is finite, there is some maximal $v$, i.e. some $v$ so that there is no $v'$ with $v'>_N v$. We now consider two cases, first wherein $p_x\in\mathcal{P}(\gamma(t_v,t_{v+1}))$ that is bound by exactly one $\alpha_z\cup\beta_z$ with $t_v\leq z<t_{v+1}$, and second where there is more than one such $z$. If the former, then every $\alpha_w\cup\beta_w$ that bounds $p_x$ is contained in the interior of $\alpha_{z}\cup\beta_{z}$. Therefore, in each such case $q_{w}$ and $q_{w+1}$ both lie in the interior of $\gamma(t_v,t_{v+1})$, and so there is some $w'\leq w$ so that $q_{t_v}<q_{w'}<q_{t_{v+1}}$. Then, let $u$ be such that $t_u^*\leq z < t_{u+1}^*$. If $w'<t_u^*$, there would then similarly be some $w''\leq w'$ so that $q_{t_u^*}<q_{w''}<q_{t_{u+1}^*}$, which would contradict the definition of the $t_i^*$. Thus, we must have $t_u^*\leq w' \leq w \leq t_{u+1}^*$. i.e. so that $p_x$ is in exactly one $\mathcal{C}(\gamma(t_v^*,t_{v+1}^*))$, as desired. 

    Now, consider the case where our maximal index $v$ is such that there are $t>1$ indices $t_v\leq z<t_{v+1}$ with $\alpha_z\cup\beta_z$ bounding $p_x\in\mathcal{P}(\gamma(t_v,t_{v+1}))$. Then, we can repeat the construction above using the outermost disc $\alpha_z\cup\beta_z$ bounding $p_x$ unless there is some curve $\gamma(t_{v'},t_{v'+1})$ with $p_x\in\mathcal{C}(\gamma(t_{v'},t_{v'+1}))$ but with no crossings in the interior of $\gamma(t_v,t_{v+1})$. In this case, however, $\gamma(t_v,t_{v+1})$ encircles $p_x$ more times than any other $\gamma(t_{v'},t_{v'+1})$, i.e. so that there is some $w$ so that the set $L_x=\{w' : k_{x,w'}=k_{x,w}\}$ is contained in the interval $(t_v,t_{v+1})$. However, notice then that any $t',t''\in L_x$ have that $k_{t',y}\neq k_{t'',y}$ for $p_y\in\mathcal{P}(\gamma(t_v,t_{v+1}))$, and thus each $t'\in L_x$ cannot be part of a matching pair.

    Then, without loss of generality, assume that we have some $p_x$ that is in exactly one $\mathcal{C}(\gamma(t_v,t_{v+1}))$. We claim we can use this curve to find our $1\leq t\leq m$ that has no matching pair $(t,l)$. We will first show this in the case that $p_x\not\in\mathcal{P}(\gamma(t_v,t_{v+1}))$, and then in the case that $p_x\in\mathcal{P}(\gamma(t_v,t_{v+1}))$. 

    If $p_x\not\in\mathcal{P}(\gamma(t_v,t_{v+1}))$, then we have that $k_{t_{v+1},x}-k_{t_v,x}=0$, since a simple closed curve has winding number 0 around punctures not in its interior. Then, it follows that the set $Y_x = \{t' : k_{t',x}\neq k_{t_v,x}\}$ is non-empty and contained in the interval $(t_v,t_{v+1})$. Therefore, if there is any matching pair $(t',l)$ with $t'\in Y_x$, we must also have $l\in Y_x$. However, as before notice then that if we have $k_{t',x}=k_{t'',x}$ for some $t',t''\in Y_x$, we have that $k_{t',y}\neq k_{t'',y}$ for $p_y\in\mathcal{P}(\gamma(t_v,t_{v+1}))$, and thus each $t'\in Y_x$ cannot be part of a matching pair. 

    Furthermore, if $p_x\in\mathcal{P}(\gamma(t_v,t_{v+1}))$, then $k_{t_{v+1},x}-k_{t_v,x}=\pm 1$, since a simple closed curve has winding number $\pm 1$ around punctures in its interior (with the sign determined by the orientation of the curve). Then, it follows that $k_{t_v,x}=k_{w,x}$ if $w<t_v$ and $k_{t_{v+1},x}=k_{w',x}$ if $w'>t_{v+1}$, since $p_x$ is not bound by $\alpha_w\cup\beta_w$ if $w$ does not lie between $t_v$ and $t_{v+1}$. If there is any $t_v<t'<t_{v+1}$ with $k_{t',x}\neq k_{t_v,x}$ and $k_{t',x}\neq k_{t_{v+1},x}$, then we can apply the same argument as last time and find a set $Y'_x=\{t' : k_{t',x}\neq k_{t_v,x}, k_{t',x}\neq k_{t_{v+1},x}\}$ of indices without matching pairs. If this set is empty, however, then there is exactly one index $1\leq j<m$, with $p_x\in \alpha_j\cup\beta_j$, and so $k_{i,x}\neq k_{i',x}$ for any $i<j$, $i'\geq j$. Then, if every index were to have a matching pair, every $i<j$ would have to be part of some matching pair $(i,s)$ where $s<j$. Thus, we can inductively reduce to looking at the subarc of $\beta$ that runs from $q_1$ to $q_{j-1}$, and so we are done. 
\end{proof}

\begin{customthm}{A}
    The Gassner representation $\tau_n$ is faithful. 
\end{customthm}
\begin{proof}

    By the Key Lemma, there is some $1\leq i \leq m$ with no matching pair, i.e. some term $t_1^{k_1}\ldots t_n^{k_n}$ in the Gassner polynomial that has coefficient $\pm 1$. This means that the Gassner polynomial is non-zero, and so we are done. 
\end{proof}



We can further establish simpler faithful representations of $B_n$ as a Corollary of this form, as below:




\begin{customthm}{B}
    The Gassner representation $\tau_n$ is faithful modulo $p$ for any $p>1$.
\end{customthm}
\begin{proof}
    The term in the Gassner polynomial with coefficient $\pm 1$ that we found in the proof of Theorem A is non-trivial modulo $p$ for any $p>1$, and so by Theorem \ref{modp-thm} we have the desired result. 
\end{proof}

\section{Faithfulness of $\rho_3\otimes \mathbb{Z}_p$}

Up to this point, our focus has been entirely on the Gassner representation $\tau_n$ and its modulo $p$ reductions. In this section, we will attempt to apply our methods to investigate the modulo $p$ Burau representation, and apply similar methods to prove that $\rho_3\otimes \mathbb{Z}_p$ is faithful for all $p$. While this result was originally proved in \cite{Lee-Song}, we hope that our methods here will prove to be useful in future study of the kernel of $\rho_n\otimes \mathbb{Z}_p$ for higher $n$. 

Our main tool in achieving our goal will be a sequence that we will call the {\it crossing-puncture sequence}, motivated by the planar interpretation of the Burau polynomial given in $\S 2$. The procedure described in this section uses the fact that the Burau polynomial is, in some sense, ``defined'' by the sequence of integers $|k_{i+1}-k_i|$ (where $k_{i+1}$ and $k_i$ denote the exponents of the monomials at the $(i+1)$-th and $i$-th crossings respectively), an intuition which we will now attempt to make precise. 

{\bf Crossing-puncture sequence.} As in $\S 2$, let $\alpha=\alpha_0$, $\beta=\phi(\beta_0)$. We may assume that $\alpha$ is an arc that joins the puncture $p_1$ directly to the puncture $p_2$. Suppose, then, that $\alpha$ and $\beta$ intersect $m$ times, at points $q_1, q_2, \ldots, q_m$. For each $i$, let $\alpha_{\mu_i}$ and $\beta_i$ be the subarcs of $\alpha$ and $\beta$ that go from $q_i$ to $q_{i+1}$, as in $\S 2$. 

Then, let $k$ be such that $\alpha_{\mu_i} \cup \beta_i$ bounds a $k$-punctured disc. The {\it crossing-puncture number} $C_i$ is equal to $k$ if $\beta$ is oriented clockwise around $\alpha \cup \beta$, and is equal to $-k$ otherwise. The {\it crossing-puncture sequence} of $\phi$ is the sequence $C_1, C_2, \ldots, C_m$.


Now, recall that we can express the Burau polynomial of $\phi$ as the sum of the monomials at $\alpha$ and $\beta$'s $m$ crossings, as follows:
\begin{align*}
    \mathcal{B}_\phi(t)=\sum_{i=0}^m a_i t^{k_i}, \text{ where each }a_i\in\{-1, 1\}.
\end{align*}
As discussed in $\S 2$, the exponents $k_i$ and $k_j$ are equal if and only if $C_i+C_{i+1}+\ldots+C_j=0$.

\begin{theorem}\label{th:faithful mod p}
The representation $\rho_3\otimes \mathbb{Z}_p$ is faithful for all $p>1$. 
\end{theorem}
\begin{proof}
Keeping the CP sequences in mind, our strategy will be to show that if $\phi\in B_3$ is such that $\hat{\iota}(\alpha_0, \phi(\beta_0))>0$, then the Burau polynomial $\mathcal{B}_\phi(t)$ necessarily has some term $t^{k_i}$ with coefficient $\pm 1$; such a term cannot vanish modulo $p$ for any integer $p>1$, and so this will complete the proof. 

Thus, let $M$ be such that $k_M$ be $\max(\{k_i : 0\leq i\leq m\})$; we will show that that there is no $r\neq M$ so that $k_r=k_M$. To do this, we will first need the following Lemma regarding the general form of the CP sequences of 3-braids. 

\begin{lemma}\label{turning-lemma}
If $\phi\in B_3$ and the CP sequence $C_1, C_2, \ldots, C_n$ of $\phi$ does not have $C_i=C_{i+1}$ for all $i$, it is of the form
\begin{align*}
    C_{1, 1}, C_{1, 2}, &\ldots, C_{1, i_1}, \\
    C_{2, 1}, C_{2, 2}, &\ldots, C_{2, i_2}, \\
    &\ldots\\
    C_{m, 1}, C_{m, 2}, &\ldots, C_{m, i_m}
\end{align*}
where for each $r<m$, we have $i_r\in\{2, 3\}$, $|C_{r,i_r}|=3$, and exactly one $|C_{r, j_r}|=1$. For $r=m$, we have $i_r\in\{2,3,4\}$, $|C_{r, i_r}|=3$ or $|C_{r,i_r-1}|=3$ and  exactly one $|C_{r, j_r}|=1$.
\end{lemma}
\begin{proof}
To begin, notice that if $\beta$'s first crossing has sign $+$, every disc corresponding to a $|C_i|=1$ has $q_i$ with positive sign and $q_{i+1}$ with negative sign. Furthermore, every disc corresponding to a $|C_i|=3$ has $q_i$ with negative sign, and $q_{i+1}$ with positive sign, and any disc corresponding to a $|C_i|=2$ has $q_i$ and $q_{i+1}$ with the same sign.

\begin{figure}[htpb!]\centering\label{case1}
\begin{subfigure}[c]{.43\textwidth}
    \begin{overpic}{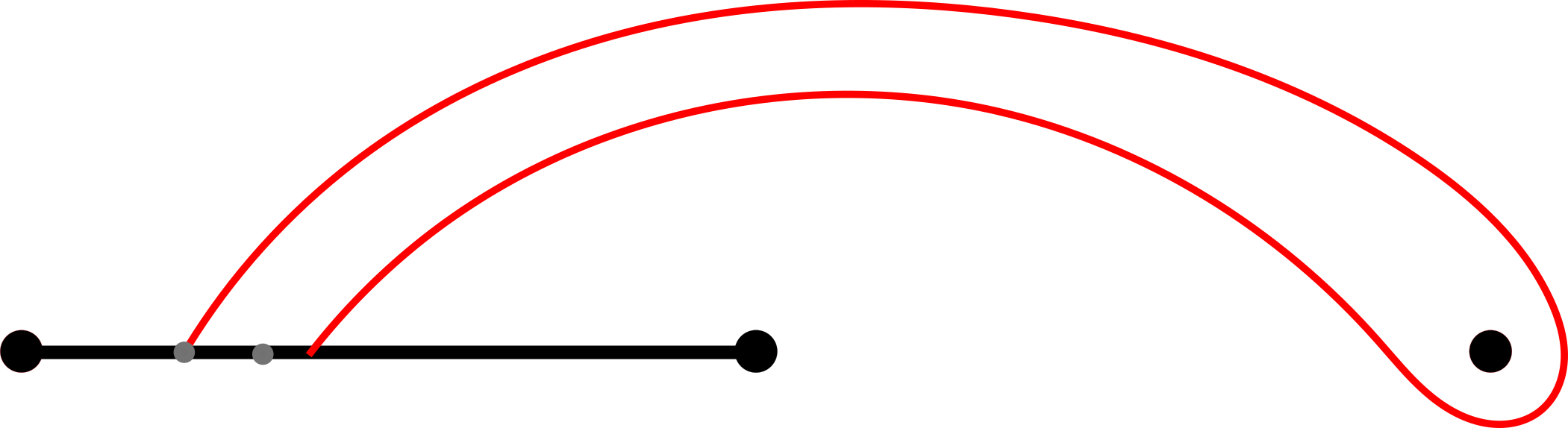}
        \put (0,9) {\tiny $p_1$}
        \put (46,9) {\tiny $p_2$}
        \put (90,9) {\tiny $p_3$}
        \put (9,1) {\tiny $q_i$}
        \put (17,1) {\tiny $q_{i+1}$}
    \end{overpic}
    \caption{The crossing $q_i$ has positive sign and $q_{i+1}$ has negative sign}%
\end{subfigure}\hspace{5mm}\begin{subfigure}[c]{.43\textwidth}
    \begin{overpic}{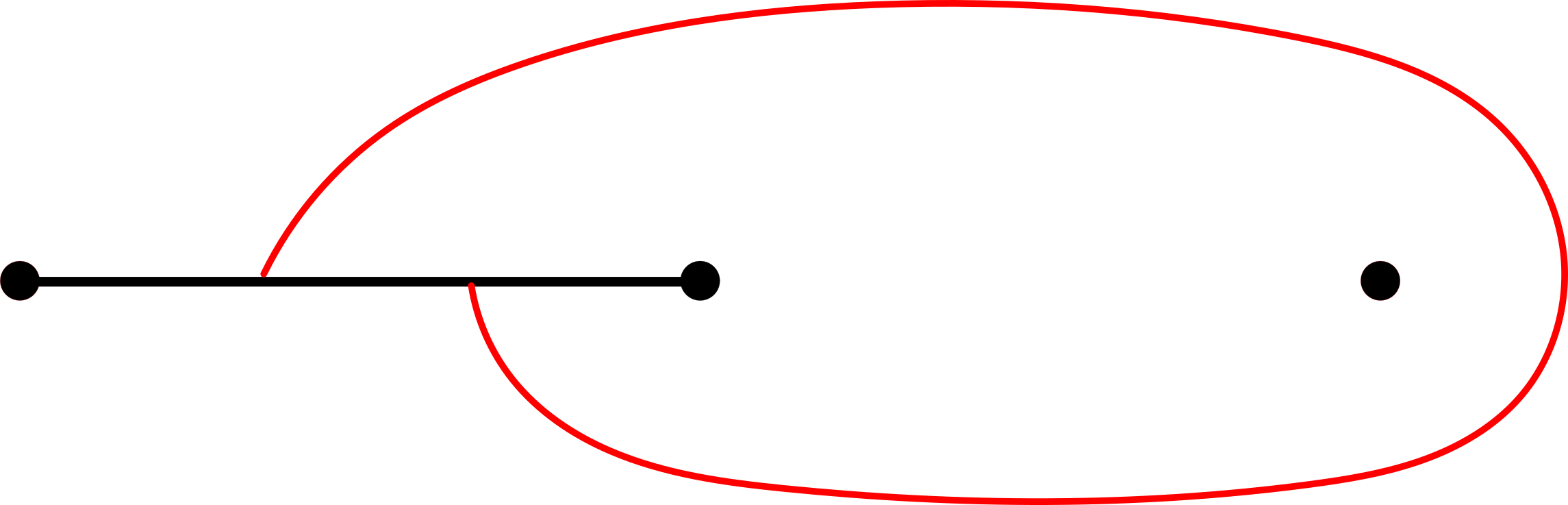}
        \put (-1,10) {\tiny $p_1$}
        \put (43,10) {\tiny $p_2$}
        \put (86,10) {\tiny $p_3$}
        \put (15,11) {\tiny $q_i$}
        \put (26,17) {\tiny $q_{i+1}$}
    \end{overpic}
    \caption{The crossings $q_i$ and $q_{i+1}$ have the same sign}%
\end{subfigure}
\caption{Discs corresponding to CP numbers of 1 and 2}
\end{figure}

It follows from this observation that if the CP sequence does not solely consist of $|C_{i}|=2$, i.e. if we do not have $C_i=C_{i+1}$ for all $i$, in order for the parity of the crossings to match up, the sequence must be a concatenation of $m$ blocks of the form $C_{r,1}\ldots,C_{r,i_r}$, where for each $r<m$, $|C_{r,i_r}|=3$ and exactly one $|C_{r,j_r}|=1$. The only claim remaining then is that each $i_r\in\{2,3\}$. To see that this is the case, notice that if $i_r>3$, there would need to be $|C_{r, t_1}|=|C_{r, t_2}|=2$ with $t_1 < j_r$ and $j_r < t_2 < i_r$. However, as we can see in the following figure, this would make it impossible for $\beta$ to extend to its endpoints at $p_*$ and $p_3$. The point here is that since $\beta$ is simple, the arc strands must stay in the same positions relative to each other. 

\begin{figure}[htpb!]
    \begin{overpic}[scale=0.9]{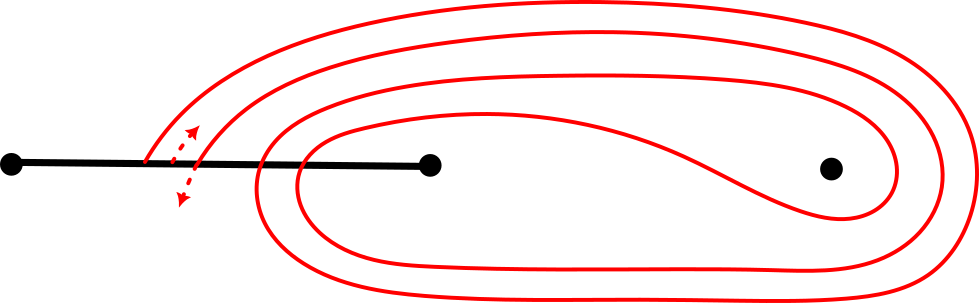}
    \end{overpic}
\caption{$\beta$ cannot reach $p_3$ without intersecting itself}
\end{figure}

This completes the proof of Lemma \ref{turning-lemma}. \end{proof}

We are now ready to complete the proof of Theorem \ref{th:faithful mod p}. To begin, we point out that it is easy to see, via a nesting argument, that if $i_{r_1}=i_{r_2}=3$, then we have $C_{r_1, 1}=C_{r_2,1}$, $C_{r_1, 2}=C_{r_2,2}$, and $C_{r_1, 3}=C_{r_2,3}$ (i.e. every block of length 3 is the same). Then, we may assume without loss of generality that the sum of this block in the CP sequence is positive: if not, we may repeat the same argument replacing $k_M$ with $\text{min}(\{k_i : 0\leq i\leq m\})$. 

Then, we can easily compute that the possible sequences taken by this block of length 3 are as follows:
\begin{itemize}
    \item 2, -1, 3
    \item -2, 1, 3
    \item 1, -2, 3
\end{itemize}

Furthermore, by Lemma \ref{turning-lemma}, every block of length 2 is of the form $\pm 1, \pm 3$. Thus, by our assumption of $k_M$'s maximality, if there is some $y \neq M$ with $k_y=k_M$, the monomial $t^{k_y}$ must a priori appear after some $C_{t, i_t}$ (i.e. at the end of a block) or after some $C_{t, 1}=1$ with $i_t=2$ and $C_{t, 2}=-3$. However, we can go even further so as to rule out the latter case: we must then have that $|C_{t+1, 1}|=2$, meaning that the next block has a positive sum, contradicting maximality. 

Then, both maximal exponents much appear at the end of some block. By maximality, the block after both $t^{k_y}$ and $t^{k_M}$ monomials must have a negative sum; thus, in order to avoid having a higher exponent appear, the only possibility is that the next block is $-1, -3$. However, as we can see in the following figure, we cannot have two $-1, -3$ blocks in the CP sequence, as this makes it impossible for $\beta$ to reach its endpoints. 

\begin{figure}[htpb!]
    \begin{overpic}{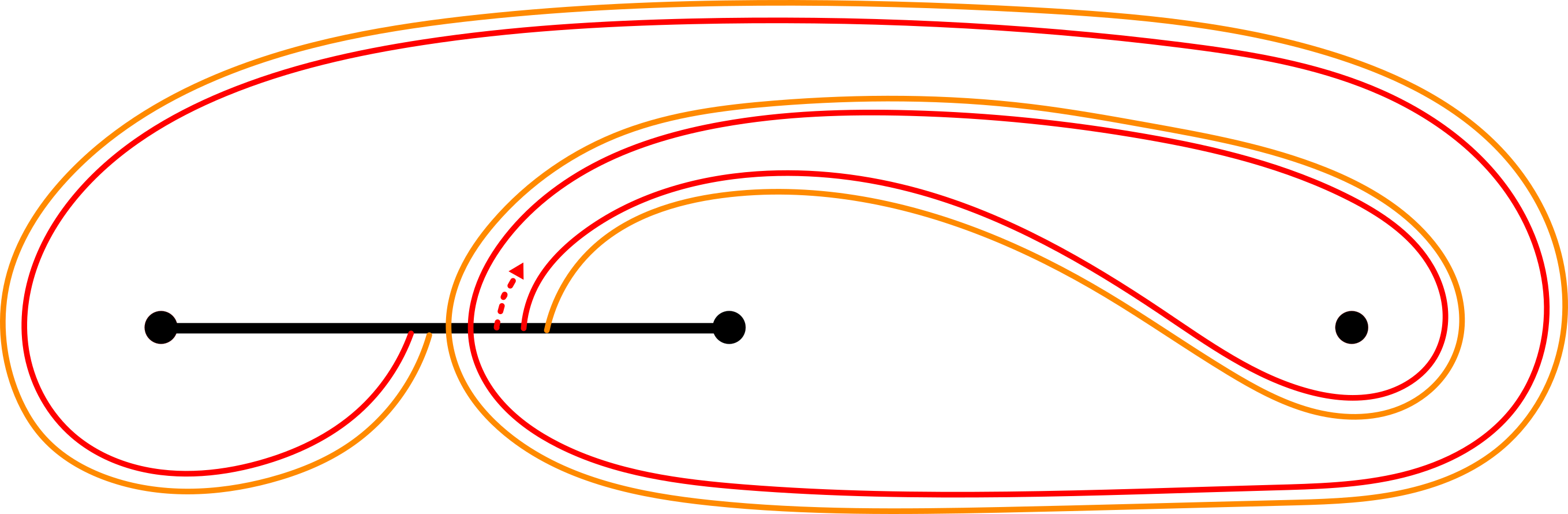}
    \end{overpic}
\caption{These cannot both be subarcs of an arc that begins at $p_*$ and ends at $p_3$}
\end{figure}

Thus, the Burau polynomial $\mathcal{B}_\phi(t)$ necessarily has some term $t^{k_M}$ with coefficient $\pm 1$, as desired. This completes the proof of Theorem \ref{th:faithful mod p}. \end{proof}




\section{A note on pseudo-Anosov mapping classes}

The foundational Thurston-Nielsen classification of diffeomorphisms of surface \cite{Thu88} allows us to classify mapping classes of surfaces into three categories: periodic, reducible, and pseudo-Anosov elements. In the context of braid groups, a periodic mapping is a braid that has some power which lies in the centre $Z(B_n)$, and a reducible mapping is a braid that fixes a family of disjoint essential curves.


It is known via \cite{Lee-Song} that $\ker(\rho_4 \otimes \mathbb{Z}_p)$ is entirely pseudo-Anosov for all $p$, and so we point out that one might perhaps attempt to understand $\ker(\rho_4\otimes \mathbb{Z}_p)$ via some sort of invariant of the curve complex that is encoded in $\rho_4\otimes \mathbb{Z}_p$. While we do not suggest here what such an invariant might look like in the case of $n=4$, we briefly show how such an idea can be used to show that $\rho_3$ acts faithfully on pseudo-Anosov 3-braids. 

\begin{theorem}
Given $\phi\in B_3$, the matrix $\rho_3(\phi)$ encodes the action of $\phi$ on an essential closed curve $\gamma'$.
\end{theorem}

\begin{proof}

We point out that by showing that the action of $\phi\in B_3$ on the essential closed curve $\gamma'$ is encoded in its image under $\rho_3$, we also show that that if $\phi$ is pseudo-Anosov (and therefore not reducible), it has a non-trivial image under the Burau representation. 

\begin{figure}[htpb!]
    \centering
    \begin{overpic}[scale=.9]{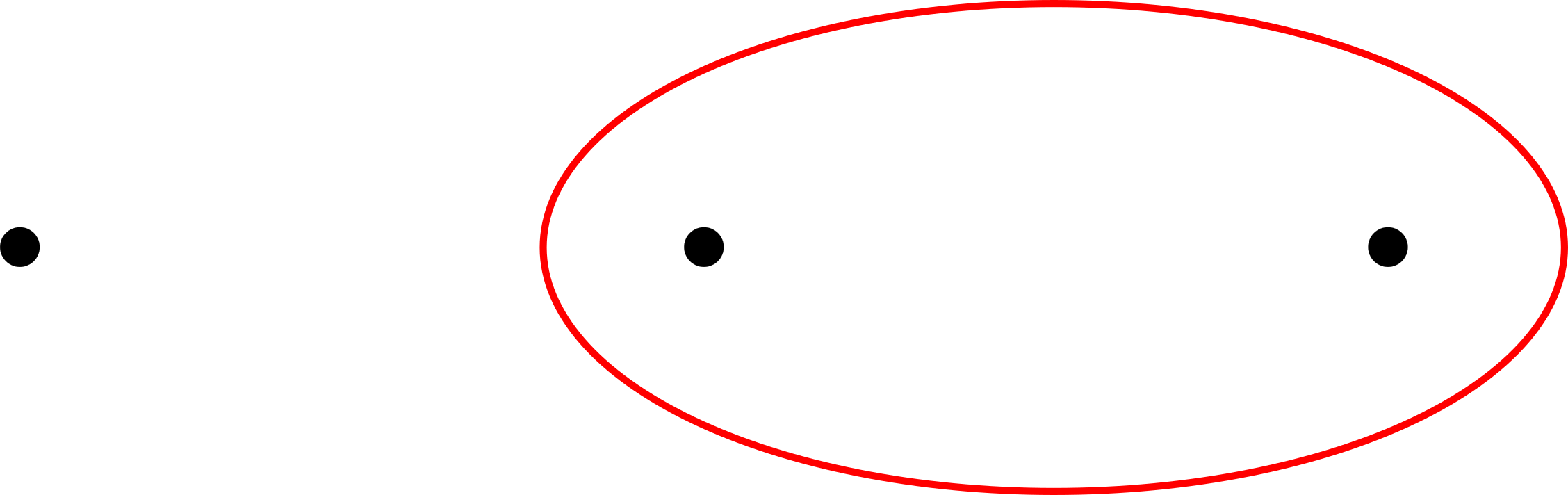}
    \end{overpic}
    \caption{If the braid $\phi$ is pseudo-Anosov, then it does not fix $\gamma'$}
\end{figure}

The key tool in this proof will be a new application of a well-known invariant of the mapping class group of a 3-times punctured disc. See the discussion about the Farey graph in \cite{Ser85}, for example.  The invariant, defined below, will be a complete invariant of the curve $\phi(\gamma')$ as well as define an invariant in the image $\rho_3(B_3)$, thus proving that $\rho_3(\phi)$ is non-trivial.

Without loss of generality, we may assume that $\phi=\sigma_1^{a_1}\sigma_2^{b_1}\ldots\sigma_1^{a_n}$, and that each $|a_i|, |b_i| >0$; this expression allows us to define a map $\pi: B_3 \to \mathbb{Q}\cup\{\infty\}$ that takes $\phi$ to the continued fraction $[a_1, -b_1, \ldots, -b_{n-1}, a_n]$, i.e.
\begin{align*}
\pi(\phi) = a_1+\cfrac{1}{-b_1+\cfrac{1}{a_2+\cfrac{1}{\ddots+\cfrac{1}{a_n}}}}
\end{align*}

Since the braid relation corresponds to equality in $\mathbb{Q}\cup\{\infty\}$ under $\pi$, the map gives a 3-braid invariant as desired. The following property of $\pi$ is well-known:
\begin{proposition}
The value $\pi(\phi)\in \mathbb{Q}\cup\{\infty\}$ completely defines the isotopy class of the curve $\phi(\gamma')$.
\end{proposition}

Then, if $\phi$ is pseudo-Anosov and does not fix $\gamma'$, the rational number $\pi(\phi)$ is non-zero. 

In order to see that this implies $\rho_3(\phi)\neq I$, we first recall the well-known fact that the braid group $B_3$ is the universal central extension of the modular group $\text{PSL}(2, \mathbb{Z})$.  In fact, the natural homomorphism from $B_3 \to \text{PSL}(2, \mathbb{Z})$ factors through the reduced Burau representation, with $\rho_3$  taking the generators of $B_3$, $\sigma_1$ and $\sigma_2$, to the following matrices, which when evaluated at $t=-1$, lift to $S$ and $T$, generators of PSL$(2, \mathbb{Z})$: 

\begin{align*} 
\rho_3(\sigma_1) = \begin{pmatrix}-t & 1\\ 0 & 1 \end{pmatrix}, \;
\rho_3(\sigma_2) = \begin{pmatrix} 1 & 0\\ t & -t \end{pmatrix}
\end{align*}

\begin{align*}
S = \begin{pmatrix}1 & 1\\ 0 & 1 \end{pmatrix}, \; 
T = \begin{pmatrix} 1 & 0\\ -1 & 1 \end{pmatrix}
\end{align*}

Now, recall that the group $\text{PSL}(2, \mathbb{Z})$ can be viewed as the group of integral linear fractional transformations $\mathbb{Q}\cup\{\infty\}\to\mathbb{Q}\cup\{\infty\}$ given by $z\to \frac{az+b}{cz+d}$, where $a,b,c,d \in \mathbb{Z}.$ Then, any continued fraction expansion $[c_0,  \ldots, c_n]$ induces a function $\mathbb{Q}\cup\{\infty\}\to\mathbb{Q}\cup\{\infty\}$ corresponding to an element of $\text{PSL}(2, \mathbb{Z})$ by taking $z \to [c_0, . . . , c_n + z]$. 

We can further note (for example through the results in \cite{Ser85}) that the continued fraction $\pi(\phi)=[a_1, -b_1, \ldots, a_n]$ corresponds exactly to the element $S^{a_1}T^{b_1}\ldots S^{a_n}$, i.e. the image of $\rho_3(\phi)$ evaluated at $t=-1$. Therefore, if $\phi$ is pseudo-Anosov, and $\pi(\phi)\neq 0$, the matrix $\rho_3(\phi)$ is non-trivial. \end{proof}

\section{Open problems}

We will end this note by pointing out some other possible questions for further research, relating to the Gassner and Burau representations of $B_n$ for various values of $n$.

\begin{open-problem}{1}\normalfont
At which specialisations of $t$ is $\rho_3$ faithful? What about $\tau_n$? Are there any faithful real or rational specialisations of $t_1,\ldots, t_n$ for $\tau_n$?
\end{open-problem}

We point out that in \cite{scherich}, the {\it real} discrete specializations of the indeterminate $t$ at which $\rho_3$ is faithful are completely classified. The question regarding rational specialisations of $\tau_n$ seems to be very pertinent as the question of whether or not the braid groups are $\mathbb{Q}$-linear is wide open; it seems possible that a more sophisticated version of our argument could be used to show that the Gassner polynomial does not have roots at $t_1=p_1^{-1}, \ldots, t_n=p_n^{-1}$ for $p_1,\ldots,p_n$ prime numbers. 

\begin{open-problem}{2}\normalfont

Is $\rho_4\otimes \mathbb{Z}_p$ faithful for any $p\geq 5$?
\end{open-problem}

\begin{open-problem}{3}\normalfont
In \cite{church-farb} it was shown that $\ker(\rho_n)$ is infinitely generated for $n\geq 6$. For which values of $n, p$ is $\ker(\rho_n\otimes \mathbb{Z}_p)$ infinitely generated? Can you give a geometric description of the kernel using the Burau polynomial? 
\end{open-problem}


\begin{open-problem}{4}\normalfont
Generalise \cite{CL97}'s presentation for $\ker(\rho_n\otimes \mathbb{Z}_2)$ to $\ker(\rho_n\otimes \mathbb{Z}_p),p>2$. 
\end{open-problem}

\begin{open-problem}{4}\normalfont
What forms can the crossing puncture sequences $C_1, \ldots, C_m$ take for higher $n$? Is this question  decidable?
\end{open-problem}

\bibliography{bib.bib}{}
\bibliographystyle{alpha}

\textsc{Vasudha Bharathram\\
Student, Princeton University\\
6484 Frist Campus Center \\
Princeton, NJ 08544, USA}\\
vb8046@princeton.edu

\end{document}